\documentclass[11pt]{article}

\usepackage{graphicx,fixmath}
\usepackage{mathdots,centernot}
\usepackage{amsmath,amscd,amsthm,amssymb,latexsym,enumerate}
\usepackage{humanist,nccmath,xfrac}
\usepackage{relsize,enumitem}
\usepackage[top=1in, bottom=1in, left=1.2in, right=1.2in]{geometry}
\usepackage{fancyhdr}
\usepackage{changepage}
\newcommand{\ra}{\rightarrow}
\newcommand{\da}{\downarrow}

\newcommand{\C}{\mathbb{C}}
\newcommand{\R}{\mathbb{R}}
\newcommand{\N}{\mathbb{N}}
\newcommand{\Z}{\mathbb{Z}}
\newcommand{\F}{\mathbb{F}}
\newcommand{\TT}{\mathbb{T}}
\newcommand{\DD}{{\mathbb D}}

\newcommand{\E}{\texthmin{E}}
\newcommand{\cm}{{\bf{x}}}
\newcommand{\icm}{{\bf{z}}}
\newcommand{\T}{{\mathcal T}_\E}

\newcommand{\Om}{\mathrm{\Omega}}

\newcommand{\Si}{\Sigma}
\newcommand{\si}{\sigma}

\newcommand{\al}{\alpha}
\newcommand{\be}{\beta}
\newcommand{\Ga}{{\Gamma}}
\newcommand{\ga}{\gamma}

\newcommand{\de}{\delta}

\newcommand{\Te}{\mathrm{\Theta}}

\newcommand{\eps}{\varepsilon}
\newcommand{\ep}{\epsilon}
\newcommand{\ess}{\text{\rm{ess}}}
\newcommand{\ac}{\text{\rm{ac}}}
\newcommand{\s}{\text{\rm{s}}}
\newcommand{\nt}{\text{\rm{int}}}
\newcommand{\Wr}{\operatorname{Wr}}

\newcommand{\im}{\operatorname{Im}}
\newcommand{\re}{\operatorname{Re}}

\newcommand{\diam}{\text{\rm{diam}}}

\newcommand{\Sz}{\text{\rm{Sz}}}
\newcommand{\pl}{{{}_+}\hspace{-0.04cm}}
\newcommand{\mn}{{{}_-}\hspace{-0.04cm}}
\newcommand{\plm}{{{}_\pm}\hspace{-0.04cm}}

\DeclareMathOperator{\cvh}{cvh}
\DeclareMathOperator{\ca}{Cap}
\DeclareMathOperator{\supp}{supp}
\DeclareMathOperator*{\slim}{{\it{s}}-lim}

\theoremstyle{plain}
 \newtheorem{theorem}{Theorem}[section]
 \newtheorem{proposition}{Proposition}[section]
 \newtheorem{lemma}{Lemma}[section]
 \newtheorem{corollary}{Corollary}[section]

\theoremstyle{definition}
 
 \newtheorem{definition}[theorem]{Definition}

\numberwithin{equation}{section}
\numberwithin{theorem}{section}
\numberwithin{proposition}{section}
\numberwithin{lemma}{section}

\begin{document}

\title{\bf{DYNAMICS IN THE SZEG\H{O} CLASS \\ AND POLYNOMIAL ASYMPTOTICS} \\
\bigskip
\begin{center}
\small\emph{By}
\vspace{-0.5cm}
\end{center}
}
\vspace{-5cm}
\author{\bf JACOB S. CHRISTIANSEN}
\date{}

\maketitle

\begin{adjustwidth*}{1.6cm}{1.6cm}
{\bf{Abstract.}}
We introduce the Szeg\H{o} class, $\Sz(\large\E)$, for an arbitrary Parreau--Widom set $\large\E\subset\R$ and study the dynamics of its elements under the left shift. When the direct Cauchy theorem holds on $\C\setminus\large\E$, we show that to each $J\in\Sz(\large\E)$ there is a unique element $J'$ in the isospectral torus, $\T$, so that the left-shifts of $J$ are asymptotic to the orbit $\{J'_m\}$ on $\T$. Moreover, we show that the ratio of the associated orthogonal polynomials has a limit, expressible in terms of Jost functions, as the degree $n$ tends to $\infty$.
This enables us to describe the large $n$ behaviour of the orthogonal polynomials for every $J$ in the Szeg\H{o} class. \\



\end{adjustwidth*}

\flushbottom

\vspace{1cm}

\section{Introduction}
\label{intro}
The present paper deals with the dynamics of one- and two-sided Jacobi matrices. Given $J=\{a_n,b_n\}_{n=1}^\infty$, we can shift the parameters $m$ $(\geq 1)$ steps to the left in order to get a new Jacobi matrix, denoted $J_m$. Rather than deleting the first $m$ $a$'s and $b$'s, we think of $J_m$ as representing an operator on $\ell^2(\{-m+1,-m+2,\ldots\})$ which is naturally embedded in $\ell^2(\Z)$.

When the parameters $a_n, b_n$ are bounded, the sequence $\{J_m\}$ has (by compactness) at least one and often several accumulation points in the strong topology. Every such point $J'$, by nature a two-sided Jacobi matrix, is called a \emph{right limit} of $J$. If we write $J'=\{a_n',b_n'\}_{n=-\infty}^\infty$, then for some subsequence $m_l\to\infty$ and all fixed $n\in\Z$,
\begin{equation}
\label{rl}
a_{n+m_l}\to a_n', \quad b_{n+m_l}\to b_n'.
\end{equation}

The concept of right limits was introduced as a tool in spectral analysis by Last and Simon \cite{MR1666767} (see also \cite[Chap.~7]{MR2743058}). A very basic result is that
\begin{equation}
\label{ess}
\si(J')\subset\si_\ess(J),
\end{equation}
where the essential spectrum ($\si_\ess$) by definition is all of $\si(J)$, except isolated eigenvalues. More importantly, as proven in \cite{MR2254485}, the union of spectra of all right limits is equal to $\si_\ess(J)$. As regards the absolutely continuous part of the spectral measures, we have the opposite inclusion
\begin{equation}
\label{ac}
\Si_\ac(J')\supset\Si_\ac(J)
\end{equation}
for essential supports. Indeed, $J'$ has a.c. spectrum of multiplicity two on $\Si_\ac(J)$. When combining \eqref{ess} and \eqref{ac}, it readily follows that if $\si_\ess(J)=\Si_\ac(J)$ then also $\si(J')=\Si_\ac(J')$ and all four sets coincide.

In addition to the above, any right limit of $J$ is \emph{reflectionless} on $\Si_\ac(J)$. This important result, due to Remling \cite{MR2811596}, will play a key role for us. By definition, a two-sided Jacobi matrix $J'$ is said to be reflectionless on a set $\large\E\subset\R$ if (for all $n$) the diagonal spectral theoretic Green's function
\begin{equation}
\label{reflectionless}
G_{nn}(t+i0):=\lim_{\eps\da 0}\bigl<\de_n, (J'-t-i\eps)^{-1}\de_n\bigr>
\end{equation}
is purely imaginary for Lebesgue a.e. $t\in\large\E$. As explained in \cite{MR2594337}, there are several equivalent definitions of `reflectionless'. We shall bring some of them into play along way (in Section \ref{sec2}). As for now, we just mention that the name goes back to Craig \cite{MR1027503} and is used in scattering theory to describe vanishing reflection coefficients.
See, e.g., \cite{MR1414303,MR2430447,MR2271928,MR2467016,MR2504863,MR2824839,MR2775395,MR3027547,MR1674798} for further information and \cite{MR1896882} for more details on scattering theory for Jacobi matrices.

Given a compact set $\large\E\subset\R$, we denote by $\T$ the set of all two-sided Jacobi matrices that are reflectionless on $\large\E$ and have spectrum equal to $\large\E$. This set is of no interest when $\vert\large\E\vert=0$ (i.e., has zero Lebesgue measure) as the condition of being reflectionless on $\large\E$ then becomes vacuous. But when $\vert\large\E\vert>0$, the set $\T$ is of paramount importance as reaffirmed by the Denisov--Rakhmanov--Remling theorem:
\begin{theorem}
\label{DRR}
Suppose that $\vert\large\E\vert>0$ and assume also that $\large\E$ is essentially closed. If $J=\{a_n,b_n\}_{n=1}^\infty$ is a Jacobi ma\-trix with
\begin{equation}
\label{ess ac}
\si_\ess(J)=\Si_\ac(J)=\large\E,
\end{equation}
then any right limit of $J$ belongs to $\T$.
\end{theorem}
The extra assumption of $\large\E$ being essentially closed (i.e., it intersects an arbitrary open interval either in a set of positive Lebesgue measure or not at all) can be removed if one modifies the definition of $\T$ to also include all two-sided Jacobi matrices that are reflectionless on $\large\E$ but for which the spectrum is merely a subset of $\large\E$. However, as explained in Section \ref{sec2}, all sets $\large\E$ of consideration in the present paper will automatically be essentially closed.

It follows from the above theorem that the sequence $\{J_m\}$ of left-shifts approaches $\T$ \emph{as a set} when $m\to\infty$. We can thus think of $\T$ as the natural limiting object associated with $\large\E$ and it is natural to ask about its structure. How can we view it as a topological space and do the individual elements share certain properties?

Under fairly weak conditions on $\large\E$, one can show that $\T$ is homeomorphic to a torus (see, e.g., \cite{MR3027547} for details). In fact, Remling raised the question if there at all exist compact sets of positive measure for which this is \emph{not} the case.
For a large class of compact sets, namely when $\large\E$ is a Parreau--Widom set and the direct Cauchy theorem holds on $\C\setminus\large\E$, the homeomorphism can be described quite explicitly. In the seminal work \cite{MR1674798}, which deals with such Parreau--Widom sets, Sodin and Yuditskii proved that $\T$ is homeomorphic to the $\ell$-dimensional torus, $\TT^\ell$, where $\ell$ is the number (possibly infinite) of bounded components in $\R\setminus\large\E$. Moreover, the Jacobi parameters of elements in $\T$ are uniformly almost periodic sequences (i.e., almost periodic in the Bohr/Bochner sense).

We shall introduce the notion of Parreau--Widom sets in Section \ref{sec2} and explain why the direct Cauchy theorem is needed in
Section \ref{sec3}. At this point, we just mention that the notion includes all compact sets that are 
so-called \emph{homogeneous} in the sense of Carleson \cite{MR730079} (e.g., Cantor sets of positive measure).

While Theorem \ref{DRR} applies to a wide class of Jacobi matrices associated with $\large\E$, it does not give details on \emph{how} the sequence $\{J_m\}$ approaches $\T$. In this paper, we shall assume more than \eqref{ess ac} but so much the stronger the conclusion will be. Our first main result, valid when $\large\E$ is a Parreau--Widom set and the direct Cauchy theorem holds, reads as follows:
\begin{theorem}
\label{JSC}
Suppose that $J=\{a_n,b_n\}_{n=1}^\infty$ belongs to the Szeg\H{o} class for $\large\E$. Then there is a unique $J'=\{a_n', b_n'\}_{n=-\infty}^\infty$ in $\T$ such that
\begin{equation}
\label{limit}
\vert a_n-a_n' \vert + \vert b_n-b_n'\vert \to 0 \; \mbox{ as $n\to\infty$}.
\end{equation}
Moreover, $\prod_{n=1}^\infty (a_n/a_n')$ converges conditionally as does $\sum_{n=1}^\infty (b_n-b_n')$.
\end{theorem}
We shall define the Szeg\H{o} class for $\large\E$ in Section \ref{sec2}. It consists of all those $J$ with essential spectrum $\large\E$ for which the so-called Szeg\H{o} condition holds and where the eigenvalues in $\R\setminus\large\E$ obey a certain Blaschke condition.

The left-shifts of $J$ have an actual limit precisely when $\large\E$ is an interval ($\T$ consists of a single point in that case).
As soon as there is a gap in $\large\E$, $J$ has at least two different right limits. The strength of Theorem \ref{JSC} lies in providing a unique element $J'\in\T$ so that the left-shifts of $J$ are asymptotic to the orbit $\{J_m'\}$ on $\T$. Specifically, the condition \eqref{limit} is equivalent to
\begin{equation}
J_m-J_m'\xrightarrow[]{\,str.\,} 0.
\end{equation}
As the proof will reveal, one can single out the unique $J'$ by matching the characters of the associated Jost functions (to be defined in Section \ref{sec3}). In exact language, this simply means that $\chi(J')=\chi(J)$.

The conditional convergence of $\prod(a_n/a_n')$ and $\sum(b_n-b_n')$ is a direct consequence of Szeg\H{o} asymptotics for the associated orthonormal polynomials (to be established in Section \ref{sec5}). We conjecture that one also has
\begin{equation}
\label{ell2}
\medop\sum_{n=1}^\infty(a_n-a_n')^2+(b_n-b_n')^2<\infty
\end{equation}
in case of which $\sum(a_n-a_n')$ converges conditionally too.
For comparison, the stronger condition
\begin{equation}
\label{ell1}
\medop\sum_{n=1}^\infty|a_n-a_n'|+|b_n-b_n'|<\infty
\end{equation}
implies Szeg\H{o} asymptotics as in Theorem \ref{Sz asympt} below. One can show this by a simple variation of parameters argument using a discrete version of Levinson's theorem (see, e.g., \cite{MR0156122} or \cite{MR1002291}). See also \cite{MR1145921} for a different approach.

Based on \cite{MR2785827}, it is natural to conjecture that \eqref{ell1} in fact implies that $J$ belongs to the Szeg\H{o} class for $\large\E$. But one should bear in mind that Szeg\H{o} asymptotics can occur outside the Szeg\H{o} class (see, e.g., \cite{MR1091035,MR2221136} or \cite{MR2891229}). We seek to investigate this conjecture and related issues in an upcoming paper.

Results similar to Theorem \ref{JSC} (and Theorem \ref{Sz asympt} below) have been obtained by Peherstorfer and Yuditskii
\cite{MR1981915,PY} in the framework of homogeneous sets. However, our approach is centered around $\T$ and the method of
proof is very different from the variational approach used in \cite{MR1981915}.
Our proof of \eqref{limit} is inspired by the techniques developed in \cite{MR2659747,MR2784484} 
and relies on three important results:
\begin{itemize}
\item[(1)]
the Denisov--Rakhmanov--Remling theorem (Theorem \ref{DRR}) \\
 --- any right limit of $J$ belongs to $\T$,
\item[(2)]
the Jost isomorphism of Sodin and Yuditskii (Theorem \ref{SY}) \\
 --- the map $\T\ni J'\mapsto\chi(J')\in\Ga_\E^{\,*}$ is a homeomorphism,
\item[(3)]
the Jost asymptotics for right limits (Theorem \ref{Jost asympt}) \\
 --- if $J_{m_l}\xrightarrow[]{\,str.\,} J'\in\T$, then $\chi(J\vert_{m_l})\to\chi(J')$.
\end{itemize}
We shall expand on (2) in Section \ref{sec3} and provide a proof of (3) in Section \ref{sec4}.
The bottom line is that with (1)--(3) at our disposal, the proof reads as follows (if we use the notation of Sections 3--4):
\begin{proof}
Clearly, there can be at most one $J'$ with almost periodic Jacobi parameters so that \eqref{limit} holds. Let $J'$ be the unique element in $\T$ for which
\begin{equation}
\label{char eq}
\chi(J')=\chi(J).
\end{equation}
As a direct consequence of \eqref{linearize}, we have $\chi(J'_m)=\chi(J\vert_m)$ for all $m\geq 1$.
Assume now that $|a_n-a_n'|+|b_n-b_n'|\not\rightarrow 0$. Our goal is then to arrive at a contradiction.

If \eqref{limit} fails, there is a subsequence $m_l\to\infty$ so that $J$ and $J'$ have different right limits along $m_l$, say
\begin{equation}
\label{strong}
\slim_{\quad\! l\to\infty} J_{m_l}=K \, \neq \, K'=\slim_{\quad\! l\to\infty} J'_{m_l}.
\end{equation}
According to the Denisov--Rakhmanov--Remling theorem, both $K$ and $K'$ lie in $\T$.
Moreover, as \eqref{strong} leads to locally uniform convergence of the associated Jost functions (by Theorem \ref{Jost asympt}), we have
\begin{equation}
\chi(J\vert_{m_l}) \ra \chi(K)
\; \mbox{ and } \;
\chi(J_{m_l}') \ra \chi(K'). 
\end{equation}
Consequently, $\chi(K)=\chi(K')$ and therefore $K=K'$ by the Jost isomorphism (Theorem \ref{SY}). But this contradicts \eqref{strong} and thus \eqref{limit} must be true.
\end{proof}


To formulate the statement on Szeg\H{o} asymptotics we first introduce some notation. Given a bounded Jacobi matrix $J=\{a_n, b_n\}_{n=1}^\infty$ with spectral measure $d\mu$, let $P_n(x):=P_n(x,d\mu)$ denote the associated \emph{orthonormal polynomials} (i.e., $P_n$ is of degree $n$, has positive leading coefficient, and $\int P_n P_m \,d\mu=\de_{n,m}$).
As is well known,
\begin{equation}
\label{3term}
xP_n(x)=a_{n+1}P_{n+1}(x)+b_{n+1}P_n(x)+a_nP_{n-1}(x)
\end{equation}
for all $n\geq 1$ (or $n\geq 0$ if we define $a_0\equiv 0$). When $J'$ belongs to $\T$, we write $d\mu'$ for the spectral measure of $J'$ restricted to $\ell^2(\N)$. This restriction of $J'$ is also denoted by $J^+$ in Section \ref{sec2} and onwards.

\begin{theorem}
\label{Sz asympt}
Suppose that $J$ belongs to the Szeg\H{o} class for $\large\E$ and let $J'$ be the unique element in $\T$ for which \eqref{limit} holds. Then
\begin{equation}
\label{ratio P}
{P_n(x,d\mu)}/{P_n(x,d\mu')}
\end{equation}
converges locally uniformly in $\overline{\C}\setminus\cvh(\large\E)$ as $n\to\infty$ and the limit is nonzero at every point $x\notin\si(J)$.
\end{theorem}
Here, $\cvh(\cdot)$ denotes the convex hull. As we will show in Section \ref{sec5}, the limit of \eqref{ratio P} can be expressed as a ratio of Jost functions. This also explains why the limit only vanishes at the isolated points in the support of $d\mu$.
With the above theorem at hand, we can now complete the proof of Theorem \ref{JSC}.
\begin{proof}[Proof, continued]
Recall that the orthonormal polynomials generated by \eqref{3term} can be written as
\[
P_n(x)=(a_1\cdots a_n)^{-1}\Bigl( x^n+(b_1+\ldots+b_n)x^{n-1}+\,\mbox{lower order terms}\, \Bigr).
\]
Hence, we have
\[
\log\biggl(\frac{P_n(x,d\mu)}{P_n(x,d\mu')}\biggr)=
-\log\Bigl(\medop\prod_{k=1}^n ({a_k}/{a_k'})\Bigr)
+\Bigl(\medop\sum_{k=1}^n(b_k-b_k')\Bigr)x^{-1}+\mathcal{O}\bigl(x^{-2}\bigr)
\]
around $x=\infty$.
Due to the uniform convergence of \eqref{ratio P}, it follows that both $\prod_{k=1}^n (a_k/a_k')$ and $\sum_{k=1}^n (b_k-b_k')$ have a limit as $n\to\infty$.
\end{proof}

Theorems \ref{JSC} and \ref{Sz asympt} describe elements of the Szeg\H{o} class by means of comparison. The set $\T$ is much smaller than $\Sz(\large\E)$ and it carries a fair amount of structure which makes the comparison interesting. As a consequence of Theorem \ref{JSC}, for instance, the Jacobi parameters of a matrix $J\in\Sz(\large\E)$ are \emph{asymptotically almost periodic}. We can specify the almost periods by computing the harmonic measure of suitable parts of $\large\E$.

Due to \eqref{limit}, the ratio in \eqref{ratio P} is a natural object of study. As the behaviour of $P_n(\,\cdot\,,d\mu')$ is well-understood, Theorem \ref{Sz asympt} also leads to a more direct description of $P_n(\,\cdot\,,d\mu)$ for large values of $n$. We will show in Section \ref{sec5} that up to some exponentially small error, $P_n(\,\cdot\,,d\mu')$ is given by the product of the Jost functions for $J^-_n$ and $J'$ divided by the $n$th power of the fundamental Blaschke product associated with $\large\E$ and some other combination of Blaschke products related to $J'$. The asymptotic behaviour of $P_n(\,\cdot\,,d\mu)$ is obtained simply by replacing the Jost function for $J'$ with the one for $J$ (see Corollary \ref{Pn}).

\section{Parreau--Widom sets and the Szeg\H{o} class}
\label{sec2}

Let $\large\E\subset\R$ be a compact set of positive logarithmic capacity.
Assume that $\large\E$ is regular and thus has no isolated points. Then it has the form 
\begin{equation}
\label{E}
{\large\E}=\bigl[\al, \be\bigr]\setminus\medmath{\sideset{}{_j}\bigcup}(\al_j,\be_j),
\end{equation}
where $\medmath{\sideset{}{_j}\bigcup}$ is a countable union of disjoint open subintervals of $[\al, \be]$ and
\[
\al<\al_i\neq\be_j<\be \;\mbox{ for all $i, j$}.
\]
Now, let $g$ be the potential theoretic Green's function for the domain $\Om:=\overline\C\setminus\large\E$ with pole at $\infty$. While $g$ vanishes on the set $\large\E$, it is strictly concave on every gap in $\large\E$. Hence, for each $j$ there is a unique point $c_j\in(\al_j,\be_j)$ so that $g'(c_j)=0$.
\begin{definition}
The set $\large\E$ is called a \emph{Parreau--Widom} set if
\begin{equation}
\label{PW}
\sideset{}{_j}\sum g(c_j)<\infty.
\end{equation}
\end{definition}
Shortly, we shall give a more geometric interpretation of such sets and explain why they necessarily have absolutely continuous equilibrium measure (denoted $d\mu_\E$) and hence positive Lebesgue measure.
\begin{definition}
When $\large\E\subset\R$ is a Parreau--Widom set, we define the \emph{Szeg\H{o} class} for $\large\E$, denoted $\Sz(\large\E)$, to be the set of all probability measures
$d\mu=f(t)dt+d\mu_\s$, with $d\mu_\s$ singular to $dt$, for which
\begin{itemize}
\item[(1)]
the essential support is equal to $\large\E$, \vspace{0.1cm}
\item[(2)]
the absolutely continuous part satisfies the \emph{Szeg\H{o} condition}
\vspace{-0.1cm}
\begin{equation}
\label{Szego}
\int_\E \log f(t)\,d\mu_\E(t)>-\infty,
\end{equation}
\item[(3)]
the isolated mass points $\{x_k\}$ outside $\large\E$ satisfy the \emph{Blaschke condition}
\vspace{-0.1cm}
\begin{equation}
\label{ev}
\sideset{}{_k}\sum g(x_k)<\infty.
\end{equation}
\end{itemize}
\end{definition}
We shall comment on the Szeg\H{o} condition momentarily. The Blaschke condition says that the isolated mass points of $d\mu$, if infinite in number, accumulate sufficiently fast at the endpoints $\al_i$ and $\be_j$. Moreover, as a direct consequence of \eqref{PW}, we see that \eqref{ev} is automatically satisfied if $d\mu$ has no more than $n$ (with $n$ fixed) mass points in every gap in $\large\E$ and only finitely many mass points outside the interval $[\al, \be]$.

Given a set $\large\E$ as in \eqref{E}, there is a unique conformal mapping $\psi$ of the upper half-plane ($\im x>0$) onto a comb-like domain of the form
\[
\bigl\{u+iv \mid 0<v<\pi, \, u>0 \bigr\}
\setminus\medmath{\sideset{}{_j}\bigcup}\bigl\{u+iv_j \mid 0<u\leq h_j\bigr\}
\]
such that $\large\E$ is mapped to the base of the comb (i.e., the line segment between $u+iv=0$ and $i\pi$) and
\[
\psi(\al)=i\pi, \quad \psi(\be)=0, \quad \psi(\infty)=\infty.
\]
Since $\large\E$ is regular, $\psi$ has a continuous extension to $\im x\geq 0$ and the set of $j$'s for which $h_j>\de$ is finite
for every $\de>0$. It readily follows that $\re\psi$ coincides with the (potential theoretic) Green's function and the Parreau--Widom condition \eqref{PW} is thus equivalent to
\begin{equation}
\label{hj}
\sideset{}{_j}\sum h_j <\infty.
\end{equation}
Using McMillan's sector theorem, see \cite{MR0257330} or \cite{MR1217706}, one can show that the equilibrium measure of $\large\E$ is absolutely continuous (wrt.\;Lebesgue measure) if and only if
\begin{equation}
\label{sector}
\sup_j \bigl\{{h_j}/{\vert v_j-v\vert}\bigr\}<\infty \;\mbox{ for a.e. $v\in(0,\pi)$}
\end{equation}
(see, e.g., \cite{MR2964141} for details). Since \eqref{hj} clearly implies \eqref{sector}, $d\mu_\E$ is absolutely continuous whenever $\large\E$ is a Parreau--Widom set. In particular, every such set has positive Lebesgue measure and is essentially closed.

While the introduction is written in the language of Jacobi matrices, we deliberately introduced the Szeg\H{o} class via probability measures. As is well known, there is a one-one correspondence between compactly supported (nontrivial) probability measures on $\R$ and bounded Jacobi matrices. Given $d\mu$, the associated Jacobi matrix is
\[
J=\left( \begin{array}{ccccc}
b_1 & a_1 &  &  & \\
a_1 & b_2 & a_2 & &  \\
\vspace{-0.1cm}
& a_2 & b_3 & \hspace{-0.1cm} a_3 & \\
& & \hspace{0.01cm} \protect\rotatebox[origin=c]{-3}{$\ddots$} &
\hspace{0.01cm} \protect\rotatebox[origin=c]{-3}{$\ddots$} &
\hspace{0.01cm} \protect\rotatebox[origin=c]{-3}{$\ddots$} \\
\end{array} \right),
\]
where the parameters $\{a_n, b_n\}_{n=1}^\infty\in (0,\infty)^\N\times\R^\N$ are the recurrence coefficients of the orthonormal polynomials $P_n(x, d\mu)$. When $\supp(d\mu)$ is compact, these coefficients are bounded. The spectrum of $J$, viewed as a bounded operator on $\ell^2(\N)$, coincides with $\supp(d\mu)$. In fact, we have
\begin{equation}
\label{m}
m(x):=\bigl< \de_1, (J-x)^{-1}\de_1\bigr>=
\int_\R\frac{d\mu(t)}{t-x}, \quad x\in\C\setminus\supp(d\mu)
\end{equation}
and shall also refer to $d\mu$ as the spectral measure of $J$. The function in \eqref{m} is called the $m$-function for $J$.

When $d\mu$ (or $J$) belongs to $\Sz(\large\E)$, we clearly have
\begin{equation}
\label{f}
f(t)>0 \,\mbox{ for a.e.\;$t\in\large\E$}
\end{equation}
and elsewhere $f$ vanishes (a.e.). Hence, $\Si_\ac(J)=\large\E$ and the hypothesis \eqref{ess ac} of the Denisov--Rakhmanov--Remling theorem is satisfied. But of course the Szeg\H{o} condition is much stricter than \eqref{f}. Assuming that (1) and (3) both hold, \eqref{Szego} is satisfied if and only if
\begin{equation}
\label{Sz thm}
\limsup_{n\to\infty}\frac{a_1\cdots a_n}{\ca(\large\E)^n}>0,
\end{equation}
that is, the product $a_1\cdots a_n/\ca(\large\E)^n$ does \emph{not} converge to $0$. Here, $\ca(\cdot)$ denotes the logarithmic capacity. Furthermore, \eqref{Szego} implies that
\begin{equation}
\label{Szego bound}
0<\liminf_{n\to\infty}\frac{a_1\cdots a_n}{\ca(\large\E)^n}\leq \limsup_{n\to\infty}\frac{a_1\cdots a_n}{\ca(\large\E)^n}<\infty.
\end{equation}
The above equivalence, which despite the ambiguity is called \emph{Szeg\H{o}'s theorem}, was proven by the author in the general setting of Parreau--Widom sets \cite{MR2855090}.

We mention is passing that the product $a_1\cdots a_n$ is the reciprocal of the leading coefficient in $P_n(x, d\mu)$. If $\large\E$ is rescaled to have logarithmic capacity $1$, the leading coefficients in the orthonormal polynomials are thus bounded above and below.
The core issue of the present paper is to be able to say much more about the asymptotic behavior of $P_n(x, d\mu)$ as well as the parameters $a_n$, $b_n$ for elements of the Szeg\H{o} class. In order to do so, we need the additional assumption of the direct Cauchy theorem (entering the field in Section \ref{sec3}).

As follows easily from the step-by-step sum rules and eigenvalue estimates proven in \cite{MR2855090}, the Szeg\H{o} class is closed under coefficient stripping. This means that the stripped matrix $J\vert_1:=\{a_{n+1},b_{n+1}\}_{n=1}^\infty$ belongs to $\Sz(\large\E)$ whenever $J=\{a_n, b_n\}_{n=1}^\infty$ does. Let us make a mental note of this simple fact.

We end this section by showing that $\Sz(\large\E)$ in a suitable way contains $\T$, the set of reflectionless matrices defined in the introduction. To be precise, we have the following
\begin{proposition}
\label{Sz contains T}
Let $\large\E\subset\R$ be a Parreau--Widom set. Suppose that $J'=\{a_n',b_n'\}_{n=-\infty}^\infty$ is reflectionless on $\large\E$
and that $\si(J')=\large\E$. Then the spectral measure of $J^+=\{a_n',b_n'\}_{n=1}^\infty$ belongs to the Szeg\H{o} class for $\large\E$.
\end{proposition}
Before the proof, we introduce some notation and briefly discuss the underlying theory. A two-sided matrix $J'$ can be split in various ways. We shall make use of the standard splitting
\[
J'=\left( \begin{array}{ccccccc}
\vspace{-0.1cm}
\ddots & \ddots & \ddots & & & &  \\
\vspace{-0.25cm}
& a_{n-1}' & \hspace{-0.2cm} b_n' \hspace{-0.2cm} & \raisebox{-4pt}{\line(0,1){12}} & a_n' & & \\
\vspace{-0.1cm}
& \line(1,0){15} & \line(1,0){15} & & \line(1,0){15} & \line(1,0){15} & \\
\vspace{-0.1cm}
& & \hspace{-0.2cm} a_n' \hspace{-0.2cm} & \raisebox{-4pt}{\line(0,1){12}} & b_{n+1}' & a_{n+1}' \\
& & \phantom{a_n} &  & \hspace{-0.2cm}\ddots & \hspace{-0.2cm}\ddots & \ddots
\end{array} \right)
\]
and denote by $J_n^+=\{a_k',b_k'\}_{k=n+1}^\infty$ the lower right corner. Flipping the parameters in the upper left corner, we get the matrix $J_n^-=\{a_{n-k}',b_{n+1-k}'\}_{k=1}^\infty$. The reflectionless condition \eqref{reflectionless} is equivalent to
\begin{equation}
\label{refl m}
(a_n')^2 m_n^\pl(t+i0)\,\overline{m_n^\mn(t+i0)}=1
\,\mbox{ for a.e.\;$t\in\large\E$ and all $n$},
\end{equation}
where $\displaystyle{m_n^\plm}$ is the $m$-function for $J_n^\pm$ 
(see, e.g., \cite{MR2594337} or \cite[Chap.~7]{MR2743058}). This of course implies that one `side' of a reflectionsless matrix uniquely determines the other. It also follows that the set $\T$ is invariant under shifts and reflections of the Jacobi parameters. For simplicity, we henceforth write $J^\pm$, $m^\plm$, $\ldots$ instead of $J_0^\pm$, $m_0^\plm$, etc.

When $\si(J')=\large\E$, the diagonal (spectral theoretic) Green's function given by
\[
G(x):=G_{00}(x)=\bigl< \de_0, (J'-x)^{-1} \de_0\bigr>
\]
is holomorphic in  $\C\setminus\large\E$ and decays like $-1/x$ at $\infty$.
If we also assume that $J'$ is reflectionless on $\large\E$ (i.e., $J'$ belongs to $\T$), there are unique points $x_j\in[\al_j,\be_j]$ such that $G$ can be written as
\begin{equation}
\label{G}
G(x)=\frac{-1}{\sqrt{(x-\al)(x-\be)}}\prod_j\frac{x-x_j}{\sqrt{(x-\al_j)(x-\be_j)}}.
\end{equation}
This follows directly from the exponential representation of $G$ (see, e.g., \cite{MR1027503} or \cite[Chap.~8]{MR1711536} for details). As will be crucial in a moment, $m^\pl$ and $m^\mn$ are related to $G$ by
\begin{equation}
\label{m pm}
(a_0')^2 m^\pl(x)-{1}/{m^\mn(x)}=-{1}/{G(x)}.
\end{equation}

Let $\cm:\DD\to\Om$ be the \emph{universal covering map}. We describe this map in greater detail in the next section and refer for now to \cite{MR1674798} (or \cite{MR2855090}). By \cite[Thm.~D]{MR1674798}, the function
\begin{equation}
\label{G fct}
\mathcal{G}(z):=-G\bigl(\cm(z)\bigr)
\end{equation}
has bounded characteristic in $\DD$. This in particular means that the angular boundary values $\mathcal{G}(e^{i\theta})$ exist for a.e.\;$\theta$ and $\log\bigl\vert \mathcal{G}(e^{i\theta})\bigr\vert$ belongs to $L^1(\partial\DD)$.
Since the equilibrium measure is preserved under the covering map and since $G$ has purely imaginary boundary values a.e.\;on {$\large\E$}, it hence follows that
\begin{equation}
\label{log H}
\log\bigl( \im G(t+i0) \bigr) \in L^1({\large\E}, d\mu_\E).
\end{equation}
We are now ready to present the proof of Proposition 2.1.
\begin{proof}
Since it is easy to see that $\si_\ess(J^+)=\large\E$, we merely focus on \eqref{Szego}--\eqref{ev}.
Let $d\mu'=f^\pl(t)dt+d\mu_\s'$ be the spectral measure of $J^+$ and recall that
\begin{equation*}
f^\pl(t)=\mfrac{1}{\pi}\im m^\pl(t+i0) \, \mbox{ a.e.\;on $\R$}.
\end{equation*}
According to \eqref{refl m}, we also have
\begin{equation*}
(a_0')^2 f^\pl(t)=-\mfrac{1}{\pi}\im\bigl({1}/{m^\mn(t+i0)}\bigr) \,\mbox{ for a.e.\;$t\in\large\E$}.
\end{equation*}
Hence, it follows from \eqref{m pm} that
\begin{align*}
2(a_0')^2 f^\pl(t)=\mfrac{1}{\pi}\im\bigl(-1/G(t+i0)\bigr)=\mfrac{1}{\pi}\bigl(\im G(t+i0)\bigr)^{-1}
\, \mbox{ a.e.\;on $\large\E$}
\end{align*}
and by \eqref{log H}, this shows that the Szeg\H{o} condition is satisfied.

Because of \eqref{PW}, we automatically have $\sum_j g(x_j)<\infty$. So to prove that the Blaschke condition holds true, it suffices to show that $m^\pl$ can only have poles (in $\C\setminus\large\E$) at the zeros of $G$.  Using the Stieltjes expansion for $m^\pl$ and $m_1^\mn$, one can show that
\begin{equation}
\label{m og G}
m^\pl(x)G_{00}(x)=m^\mn(x)G_{11}(x).
\end{equation}
Since $G_{nn}$ is holomorphic in $\C\setminus\large\E$ for all $n$, a pole of $m^\pl$ which is \emph{not} a zero of $G_{00}$ $(=G)$ must also be a pole of $m^\mn$ and hence a zero of $1/m^\mn$. By \eqref{m pm}, it follows that any pole of $m^\pl$ in $\C\setminus\large\E$ is indeed a zero of $G$.
\end{proof}

Recalling that $\T$ is shift and reflection invariant, we immediately get the following
\begin{corollary}
When $J'$ lies in $\T$, all of the matrices $J_n^+$ and $J_n^-$ belong to the Szeg\H{o} class for $\large\E$.
\end{corollary}

\section{The Abel map and Jost functions}
\label{sec3}

In this section, we take a closer look at the structure of the set $\T$ of reflectionless matrices.
Throughout, we shall equip $\T$ with the strong operator topology (which in this case coincides with the weak topology).
The key result to be explained states that $\T$ is homeomorphic to a torus of dimension equal to the number of gaps in $\large\E$ whenever the Parreau--Widom condition \eqref{PW} is satisfied and the direct Cauchy theorem holds on $\C\setminus\large\E$. Hence, we shall also refer to $\T$ as the \emph{isospectral torus}. A major role is played by the Jost function, a character auto\-mor\-phic function defined for all elements of the Szeg\H{o} class. The above homeomorphism can be described explicitly through the character of this analytic function. We end the section with a short discussion of the situation when the direct Cauchy theorem fails.

Let $\large\E\subset\R$ be a compact set as in \eqref{E}. Following \cite{MR1674798}, we denote by $\mathcal{D}_\E$ the set of  \emph{divisors}, that is, formal sums
\[
D=\sideset{}{_j}\sum (y_j, \ep_j) \; \mbox{ with $y_j\in[\al_j, \be_j]$ and $\epsilon_j=\pm 1$}.
\]
Here, $(y_j,+1)$ and $(y_j,-1)$ are identified when $y_j$ is equal to $\al_j$ or $\be_j$. In topological terms, we can think of $\mathcal{D}_\E$ as a product of circles and shall equip this torus (of dimension equal to the number of gaps in $\large\E$) with the product topology.

Recalling the representation \eqref{G}, we can now define a map from $\T$ to $\mathcal{D}_\E$. All that is needed is to describe the image of $x_j$ for every $j$:
\begin{itemize}
\item[(1)]
If $x_j=\al_j$ or $\be_j$, we map $x_j\mapsto(\al_j,\pm 1)$, respectively $(\be_j,\pm 1)$.
\vspace{0.1cm}
\item[(2)]
If $x_j\in(\al_j,\be_j)$, then it is either a pole of $m^\pl$ or $1/m^\mn$ by \eqref{m pm}. Because of \eqref{m og G}, it cannot be a pole of both. When $x_j$ is a pole of $m^\pl$, we map $x_j\mapsto (x_j,+1)$ and when $x_j$ is a zero of $m^\mn$, we map $x_j\mapsto (x_j, -1)$.
\end{itemize}
The map defined by (1)--(2) is always continuous and surjective. However, it need not be one-to-one. Given $J'\in\T$, the function $H:=-1/G$ has an integral representation of the form
\begin{equation}
\label{H}
H(x)=x+A+\int_\R\frac{d\rho(t)}{t-x}
\end{equation}
for some constant $A\in\R$ and some positive measure $d\rho$ with compact support. Since $J'$ determines $d\rho$ uniquely, we also write $d\rho_{J'}$ to emphasize this dependence. The map from $\T$ onto $\mathcal{D}_\E$ is injective if and only if $d\rho_{J'}$ has \emph{no} singular part on $\large\E$ (i.e., $d\rho_{J'}$ is absolutely continuous on $\large\E$) for all $J'\in\T$. We refer to \cite{MR2504863,MR2824839,MR3027547} for details and just mention that the map is always a homeomorphism when $\large\E$ is so-called \emph{weakly homogeneous}, that is,
\begin{equation}
\label{w hom}
\limsup_{\de\da 0}\frac{\bigl\vert \large\E\cap(t-\de,t+\de)\bigr\vert}{2\de}>0
\; \mbox{ for all $t\in\large\E$.}
\end{equation}
This requirement is strictly weaker than Carleson's condition \cite{MR730079} for $\large\E$ to be \emph{homogeneous} (i.e.,
there exists $\eps>0$ such that $\vert\large\E\cap (t-\de, t+\de)\vert\geq\de\eps$ for all $t\in\large\E$ and all $\de<\diam(\large\E)$).

We shall now introduce yet another torus associated with the set $\large\E$. Let $\cm: \DD\to\Om$ be the universal covering map (which is onto but only locally one-to-one) and let $\Ga_\E$ be the associated Fuchsian group of M\"{o}bius transformations on $\DD$. Then, by construction,
\begin{equation}
\label{cover}
\cm(z)=\cm(w) \,\iff\, \exists\ga\in\Ga_\E: \, z=\ga(w)
\end{equation}
and in order to fix $\cm$ uniquely, we require that
\[
\cm(0)=\infty, \quad \lim_{z\to 0} z\cm(z)>0.
\]
Note that $\Ga_\E$ is isomorphic to the fundamental group $\pi_1(\Om)$ and hence a free group on as many generators as the number of gaps in $\large\E$.
A map $\chi:\Ga_\E\to\partial\DD$ with the property that
\begin{equation}
\label{gamma}
\chi(\ga_1 \ga_2)=\chi(\ga_1)\chi(\ga_2)
\, \mbox{ for all $\ga_1, \ga_2$}
\end{equation}
is called a \emph{unimodular character} on $\Ga_\E$. We denote by $\Ga^{\,*}_\E$ the multiplicative group of all such characters.
Since every $\chi$ satisfying \eqref{gamma} is determined by its values on the generators of $\Ga_\E$,
we can think of $\Ga^{\,*}_\E$ as a torus of dimension equal to the number of gaps in $\large\E$. This torus is a compact Abelian group when equipped with the topology dual to the discrete one (on $\Ga_\E$).

Following \cite{MR1674798}, one can define a continuous map from $\mathcal{D}_\E$ onto $\Ga^{\,*}_\E$, the so-called \emph{Abel map}. Traditionally this is done via Abelian integrals but here we shall use a different, yet equivalent, definition which is based on Jost functions.
In order to define the Jost function for measures in the Szeg\H{o} class, we first recall the notion of Blaschke products.

As $\Ga_\E$ is of convergent type (see, e.g., \cite{MR0414898}), the Blaschke products defined by
\begin{equation}
\label{Blaschke}
B(z,w)=\prod_{\ga\in\Ga_\E}\frac{\vert\ga(w)\vert}{\ga(w)}
\frac{\ga(w)-z}{1-\overline{\ga(w)}z}
\end{equation}
converge for all $z, w\in\DD$. Evidently, $B(\,\cdot\,, w)$ is analytic on $\DD$ and it has a simple zero at $\ga(w)$ for every $\ga\in\Ga_\E$. Since $\vert B(\,\cdot\,, w)\vert$ is invariant under the action of $\Ga_\E$, it follows that $B(\,\cdot\,, w)$ is \emph{character automorphic}, that is,
\begin{equation}
\label{B char}
B(\ga(z),w)=\chi_{\substack{ \\ w}}(\ga)\,B(z,w)
\end{equation}
for some character $\chi_{\substack{ \\ w}}$ in $\Ga_\E^{\,*}$.
When $w=0$, we write $B(z)$ for the Blaschke product in \eqref{Blaschke} and recall the important link
\begin{equation}
\label{B and Green}
g\bigl(\cm(z)\bigr)=-\log\vert B(z)\vert,
\end{equation}
where $g$ is the potential theoretic Greens's function (cf. Section \ref{sec2}). 

Now, fix a fundamental set $\mathbb{F}$ for $\Ga_\E$ (i.e., a subset of $\DD$ which contains one and only one point of each $\Ga_\E$-orbit). A canonical choice is the Ford fundamental region or the Dirichlet region
\[
\bigl\{z\in\DD\,:\, \vert z\vert\leq\vert\ga(z)\vert \mbox{ for all $\ga\in\Ga_\E$}\bigr\}
\]
with `half' of the boundary removed (e.g., the orthocircles in the lower half-plane).
Let $w_j$ be the unique point in $\mathbb{F}$ for which $\cm(w_j)=c_j$, where $c_j$ is the critical point of $g$ in the interval $(\al_j,\be_j)$. Because of \eqref{PW}, the product
\begin{equation}
\label{c}
c(z):=\medop{\sideset{}{_j}\prod} B(z,w_j)
\end{equation}
converges to a character automorphic function on $\DD$. We denote by $\chi_{\substack{\\ c}}$ the associated character.
\begin{definition}
Let $d\mu=f(t)dt+d\mu_\s$ be a measure in the Szeg\H{o} class for $\large\E$ and denote by $\{x_k\}$ the isolated mass points in $\R\setminus\large\E$. The \emph{Jost function} for $d\mu$ is defined by
\begin{equation}
\label{Jost}
u(z;d\mu)=\prod_k B(z,p_k)\exp\biggl\{-\int_0^{2\pi}\frac{e^{i\theta}+z}{e^{i\theta}-z}
\log\Bigl(\pi f\bigl(\cm(e^{i\theta})\bigr)\Bigr)\frac{d\theta}{4\pi}\biggr\}, \quad z\in\DD,
\end{equation}
where $p_k$ is the unique point in $\mathbb{F}$ such that $\cm(p_k)=x_k$.
\end{definition}
Due to \eqref{ev}, the \emph{Blaschke part} of $u$ converges for all $z\in\DD$. Moreover, the Szeg\H{o} condition \eqref{Szego} ensures that the integral on the right-hand side of \eqref{Jost} converges. Hence the exponential defines an outer function on $\DD$, called the \emph{Szeg\H{o} part} of $u$. Since the Blaschke and Szeg\H{o} parts are both analytic and character automorphic on $\DD$ (see, e.g., \cite{MR2659747} for details), so is $u$.

For $J'\in\T$, we denote by $u(\,\cdot\,;J')$ the Jost function for the spectral measure of $J^+$\;(i.e., the measure called $d\mu'$ in Section \ref{sec2}). Furthermore, we let $\chi(J')$ denote the associated character. It should be pointed out that $u(\,\cdot\,;J')$ is closely related to the function $k^\al$ in \cite{MR1674798}. When $\chi(J')=\al$, these two functions merely differ by a fixed character automorphic function which is independent of $d\mu'$ (but heavily depends on the equilibrium measure $d\mu_\E$). Although one may argue that it is more natural to also divide by $a_0'$, we stick for simplicity to the expression in \eqref{Jost}. Our choice of the Jost function does \emph{not} include a reference measure, as opposed to \cite{MR2659747,MR2784484}.

The seminal paper \cite{MR1674798} of Sodin and Yuditskii proved the following key result:
\begin{theorem}
\label{SY}
The map
\begin{equation}
\label{Jost iso}
\T\ni J' \longmapsto \chi(J')\in\Ga_\E^{\,*}
\end{equation}
is a homeomorphism (equivalently, the Abel map is one-to-one) if and only if the {direct Cauchy theorem} holds on $\C\setminus\large\E$.
\end{theorem}
Here, the {direct Cauchy theorem} is the statement that
\begin{equation}
\label{DCT}
\int_0^{2\pi}\frac{\varphi(e^{i\theta})}{c(e^{i\theta})}\frac{d\theta}{2\pi}
=\frac{\varphi(0)}{c(0)}
\end{equation}
whenever $\varphi\in H^1(\DD)$ is character automorphic and $\chi_{\substack{\\ \varphi}}=\chi_{\substack{\\ c}}$.
The monograph \cite{MR723502} of Hasumi gives several conditions each of which are equivalent to \eqref{DCT}. Unfortunately, there seems to be no simple geometric condition on $\large\E$ that can tell whether or not the direct Cauchy theorem holds.
It is known (see, e.g., \cite{MR871721}) that there are Parreau--Widom sets for which the direct Cauchy theorem fails. On one hand, the direct Cauchy theorem is a stronger restriction than $d\rho_{J'}$ being absolutely continuous on $\large\E$ for all $J'\in\T$. But on the other hand, it is automatically satisfied when $\large\E$ is homogeneous. The reader is referred to \cite{MR2858955,MR723502} for examples and an in-depth analysis.

As we will show in Section \ref{sec5}, the left shift on $\T$ corresponds (under the Abel map) to multiplying by the inverse of the character of $B$ on $\Ga_\E^{\,*}$. Hence, since the map of picking out a single Jacobi parameter, say $a_0'$ or $b_0'$, is continuous on $\T$, the parameter sequences $n\mapsto a_n', \, b_n'$ are uniformly almost periodic when the direct Cauchy theorem holds.
Specifically, for each $J'\in\T$ the Jacobi parameters can be obtained by evaluating a continuous function at the points $\chi_{\substack{\\ 0}}^{-n}\chi(J')$ in $\Ga_\E^{\,*}$.

We finally discuss the situation when the direct Cauchy theorem fails. In that case, Theorem \ref{JSC} no longer applies. But to which extend can we still describe the right limits of a matrix $J$ in the Szeg\H{o} class? Following \cite{VY}, we introduce the set
\[
\Te=\bigl\{\chi\in\Ga_\E^{\,*}\, : \, \chi(J')=\chi(K')=\chi \,\centernot\Longrightarrow\, J'=K' \mbox{ in $\T$} \bigr\}.
\]
That is, we single out those characters $\chi$ that correspond to more than one $J'$ in $\T$. For a wide range of Parreau--Widom sets (see, e.g., \cite{VY} for precise conditions on $\large\E$), this subset of $\Ga_\E^{\,*}$ has Haar measure zero.

For $J\in\Sz(\large\E)$, we denote by $\chi(J)$ the character of the associated Jost function. Suppose that for some subsequence $m_l\to\infty$, the characters $\chi(J)\chi_{\substack{\\ 0}}^{-m_l}$ converge to $\chi\in\Ga_\E^{\,*}$. When the direct Cauchy theorem fails, we cannot conclude that $J_{m_l}$ has a limit. But if it does, say
\[
J_{m_l}\xrightarrow[]{\,str.\,}J' \mbox{  as $l\to\infty$},
\]
then $\chi(J')$ must coincide with $\chi$ (by Theorem \ref{Jost asympt} below). A short compactness argument therefore shows that $J_{m_l}$ has a limit whenever $\chi\notin\Te$.

If $\chi\in\Te$, however, there are several $J'\in\T$ so that $\chi(J')=\chi$. In the language of \cite{VY}, we have such a $J'$ for each of the Hardy spaces between $\check{H}^2(\chi)$ and $\hat{H}^2(\chi)$. It seems unlikely that $J_{m_l}$ should have a limit in this situation, but we can always pass to convergent subsequences.
A natural question in this connection, not to be addressed here, is the following: Is it possible to obtain \emph{any} $J'$ with character $\chi$ by passing to a suitable subsequence or can we only obtain \emph{some}? If we impose extra assumptions on $\large\E$ as in \cite{VY}, the sequence $\{ \chi(J) \chi_{\substack{\\ 0}}^{-n} \}_{n\in\Z}$ is den\-se in $\Ga_\E^{\,*}$. So for every $\chi\in\Ga_\E^{\,*}$, at least one $J'$ with character $\chi$ is indeed a right limit of $J$.

\section{Jost asymptotics for right limits}
\label{sec4}

The aim is now to formulate and prove the main technical result of the paper.
Not only will this theorem on Jost asymptotics be the key to proving \eqref{limit},
it will also be an important step towards establishing po\-ly\-nomial asymptotics.

Throughout the section, $\large\E\subset\R$ will denote an arbitrary Parreau--Widom set. If $J$ has spectral measure $d\mu$, then we write $d\mu_m=f_m(t)dt+d\mu_{m,\s}$ for the spectral measure of the stripped matrix $J\vert_m$ (i.e., the restriction to $\ell^2(\N)$ of the shifted matrix $J_m$).

\begin{theorem}
\label{Jost asympt}
Let $d\mu$ be a measure in $\Sz(\large\E)$ and $J$ the associated Jacobi matrix.
If $J_{m_l}\xrightarrow[]{\,str.\,} J'$ along some subsequence $m_l\to\infty$, then
\begin{equation}
\label{Jost asymp}
u(z;d\mu_{m_l}) \ra u(z;J') \; \mbox{ as $l\to\infty$}
\end{equation}
uniformly on compact subsets of $\DD$. In particular,
\begin{equation}
\label{char conv}
\chi(J\vert_{m_l})\to\chi(J') \; \mbox{ as $l\to\infty$}.
\end{equation}
\end{theorem}
We point out that the right limit $J'$ automatically belongs to $\T$.
This is an immediate consequence of the Denisov--Rakhmanov--Remling theorem.

The proof of the theorem is divided into two steps. We start by showing that the Blaschke parts converge and then we show that also the Szeg\H{o} parts converge. As for notation, let $\{x_k\}$ be the isolated mass points of $d\mu$ outside $\large\E$.
Moreover, denote by $\{x_{l,k}\}$ and $\{x_j^\pl\}$ the eigenvalues in $\R\setminus\large\E$ of $J\vert_{m_l}$, respectively $J^+=\{a_n', b_n'\}_{n=1}^\infty$.
Finally, fix a fundamental set $\mathbb{F}$ for $\Ga_\E$ and take $p_{l,k}$ and $p_j^\pl$ to be the unique points in $\mathbb{F}$ so that $\cm(p_{l,k})=x_{l,k}$ and $\cm(p_j^\pl)=x_j^\pl$.
\begin{proposition}
\label{B conv}
Under the hypothesis of Theorem \ref{Jost asympt} and with notation as above, we have
\begin{equation}
\label{Blaschke conv}
\sideset{}{_k}\prod B(z, p_{l,k}) \longrightarrow \sideset{}{_j}\prod B(z, p_{j}^\pl) \; \mbox{ as $l\to\infty$}
\end{equation}
locally uniformly on $\DD$.
\end{proposition}
The idea of the proof is to split the product on the left-hand side into an infinite product which is within $\eps$ of $1$ and a finite product that only corresponds to points in a fixed union of finitely many gaps (in $\large\E$) for all $l$. Then we can apply the same techniques as developed for finite gap sets in \cite[Sect.~3]{MR2784484} to establish the convergence.
\begin{lemma}
\label{rho}
If $\{y_k\}$ is a collection of points in $\R\setminus\large\E$ and $\sum_k g(y_k)<\rho$, then
\begin{equation}
\Bigl\vert \sideset{}{_k}\prod B(z,z_k)-1 \Bigr\vert \leq \exp\Bigl\{\mfrac{1+r}{1-r}\rho\Bigr\}-1
\; \mbox{ for $\vert z\vert \leq r$},
\end{equation}
where $\{z_k\}$ are the unique points in $\mathbb{F}$ for which $\cm(z_k)=y_k$.
\end{lemma}
\begin{proof}
By \eqref{B and Green} and \eqref{Blaschke}, we have
\[
\sum_k g(y_k)=-\log\Bigl(\prod_{k,\ga}\vert\ga(z_k)\vert\Bigr)
\geq\sum_{k,\ga}\Bigl(1-\vert\ga(z_k)\vert\Bigr).
\]
Standard estimates for Blaschke products thus give the desired inequality.
\end{proof}

\begin{proof}[Proof of Proposition \ref{B conv}]
Given $r<1$, we show that the products in \eqref{Blaschke conv} converge uniformly in $\vert z\vert\leq r$.
To any $\rho>0$, there is a finite union of gaps, say
\[
I_\rho=(\al_{j_1},\be_{j_1})\cup\ldots\cup(\al_{j_N},\be_{j_N}),
\]
so that $\sum_{j\notin\{j_1,\ldots,j_N\}} g(c_j)<\rho$ and
\begin{equation}
\label{I rho}
\sum_{k:\, x_{l,k}\in[\al,\be]\setminus I_\rho} g(x_{l,k})<\rho
\; \mbox{ for all $l$}.
\end{equation}
For we know that $\sum_j g(c_j)<\infty$ and \eqref{I rho} is obtained by combining \eqref{ev} with the fact that the stripped matrix $J\vert_{m_l}$ has at most one eigenvalue between any two consecutive eigenvalues of $J$ in a gap in $\large\E$ (see, e.g., \cite[Sect. 3]{MR2855090}).
Actually, there is a $\de>0$ so that \eqref{I rho} holds with $I_\rho$ replaced by
\begin{equation}
\label{I delta}
I_{\rho,\de}=\medop\bigcup_{n=1}^N (\al_{j_n}+\de,\be_{j_n}-\de).
\end{equation}
Hence, we can directly apply the results of \cite[Sect. 3]{MR2784484}.
Given $\eps>0$, pick $\rho>0$ so small that $2\rho(1+r)/(1-r)<\log(1+\eps/4)$.
Moreover, take $L$ so large that
\begin{equation}
\sum_{k:\, x_{l,k}\notin[\al,\be]} g(x_{l,k})<\rho
\end{equation}
and $J\vert_{m_l}$ has the same number of eigenvalues as $J^+$ in each of the $N$ intervals in $I_{\rho,\de}$ for all $l>L$.
The above lemma then implies that
\begin{equation}
\label{B estimate}
\Bigl\vert \sideset{}{_k}\prod B(z,p_{l,k})-\sideset{}{_j}\prod B(z, p_j^\pl)\Bigr\vert
<\frac{\eps}{2}+\biggl\vert\prod_{k:\, x_{l,k}\in I_{\rho,\de}} B(z, p_{l,k})-
\prod_{j:\,x_j^\pl\in I_{\rho,\de}} B(z, p_j^\pl)\biggr\vert
\end{equation}
for $\vert z\vert\leq r$ and $l>L$.
By construction, the two products on the right-hand side have the same number of factors.
So if $x_{j}^\pl\in(\al_{j_n}+\de, \be_{j_n}-\de)$ for some $n$, there is exactly one $k$ so that $\al_{j_n}+\de<x_{l,k}<\be_{j_n}-\de$
and this $x_{l,k}$ converges to $x_{j}^\pl$ as $l\to\infty$ (according to \cite[Thm. 3.9]{MR2784484}).
Hence, for $l$ sufficiently large the products on the left-hand side of \eqref{B estimate} differ by at most $\eps$ when $\vert z\vert\leq r$. This completes the proof.
\end{proof}

Now, let again $d\mu'=f^\pl(t)dt+d\mu'_\s$ denote the spectral measure of $J^+$. Following \cite{MR2855090}, we write $S(d\mu)$ for the entropy of $d\mu_\E$ relative to $d\mu$. When $d\mu=f(t)dt+d\mu_\s$ and $f(t)>0$ for a.e. $t\in\large\E$, we simply have
\[
S(d\mu)=\int_\E\log\biggl(\frac{f_\E(t)}{f(t)}\biggr)d\mu_\E(t),
\]
where $d\mu_\E=f_\E(t)dt$.
\begin{proposition}
\label{Szego parts}
Under the hypotheses of Theorem \ref{Jost asympt} and with notation as above, we have
\begin{equation}
\label{f conv}
\log f_{m_l} \hspace{0.02cm} d\mu_\E \xrightarrow[]{\;w\;} \log f^\pl \hspace{0.02cm} d\mu_\E
\; \mbox{ as $l\to\infty$}.
\end{equation}
\end{proposition}
\begin{proof}
Since $J\vert_{m_l}\xrightarrow[]{\, str.\,} J^+$, it follows that $d\mu_{m_l}\xrightarrow[]{\,w\,}d\mu'$. Hence, by a lemma of Simon and Zlato\v{s} (see \cite{MR2020274} or \cite[Sect. 6]{MR2784484}), it suffices to show that $S(d\mu_{m_l})\to S(d\mu')$. As relative entropy is weakly upper semicontinuous, we automati\-cal\-ly get
\[
\limsup S(d\mu_{m_l})\leq S(d\mu').
\]
To prove that $\liminf S(d\mu_{m_l})\geq S(d\mu')$, start by taking a subsequence $m_{l(n)}\ra\infty$ so that
\begin{equation}
\label{liminf}
S\bigl(d\mu_{m_{l(n)}}\bigr)\to\liminf S(d\mu_{m_l}).
\end{equation}
Since the sequence ${a_1\cdots a_n}/{\ca(\large\E)^n}$ is bounded above and below by some positive constants, we can assume that
\begin{equation}
\tau_{l(n)}:=\frac{a_1\cdots a_{m_{l(n)}}}{\ca(\large\E)^{m_{l(n)}}}\longrightarrow\tau>0
\; \mbox{ as $l(n)\to\infty$}.
\end{equation}
Now, for $n<q$ consider the matrix $J_{n|q}$ given by
\begin{equation*}
\label{Jq}
J_{n|q}=\left( \begin{matrix}
\vspace{0.05cm}
b_{m_{l(n)}+1} & a_{m_{l(n)}+1} &&&&& \\
a_{m_{l(n)}+1} & b_{m_{l(n)}+2} & a_{m_{l(n)}+2} &&& \\
\vspace{-0.1cm}
& \hspace{-0.3cm} \protect\rotatebox[origin=c]{10}{$\ddots$} & \hspace{-0.8cm}\protect\rotatebox[origin=c]{10}{$\ddots$} & \hspace{-0.7cm}\protect\rotatebox[origin=c]{10}{$\ddots$} && \\
\vspace{0.05cm}
&& \hspace{-0.8cm}\protect\rotatebox[origin=c]{10}{$\ddots$} & \hspace{-0.3cm}b_{m_{l(q)}} & a_{m_{l(q)}} & \\
&&& \hspace{-0.3cm} a_{m_{l(q)}} & \hspace{-0.2cm} {\bf \phantom{.}_\ulcorner}
& \hspace{-0.5cm} \raisebox{3pt}{\line(1,0){12}}  \\
&&&& \hspace{-0.2cm} \line(0,1){10} & \hspace{-0.4cm} \raisebox{2pt}{$J^+$}  \\
\end{matrix} \right).
\end{equation*}
Applying the iterated step-by-step sum rule from \cite{MR2855090}, we get
\begin{equation*}
\label{n and q}
\log\Bigl({\tau_{l(q)}}/{\tau_{l(n)}}\Bigr)=\sum_k g\bigl(x_{n|q,k}\bigr)-\sum_j g\bigl(x_j^\pl\bigr)
+\mfrac{1}{2}\Bigl(S(d\mu_{n|q})-S(d\mu')\Bigr),
\end{equation*}
where $d\mu_{n|q}$ is the spectral measure of $J_{n|q}$ and $\{x_{n|q,k}\}$ denote the associated eigenvalues in $\R\setminus\large\E$.
The idea is now first to let $l(q)\to\infty$ and only afterwards let $l(n)\to\infty$. Clearly, $d\mu_{n|q}\xrightarrow[]{\,w\,}d\mu_{m_{l(n)}}$ and $\tau_{l(q)}\to\tau$ as $l(q)\to\infty$.
It requires an extra argument to show that
\begin{equation}
\label{g conv}
\sideset{}{_k}\sum g\bigl(x_{n|q,k}\bigr) \longrightarrow \sideset{}{_k}\sum g\bigl(x_{l(n),k}\bigr).
\end{equation}
Observe that $J_{n|q}$ is a rank two perturbation of the direct sum of the upper left corner of $J\vert_{m_{l(n)}}$ and $J^+$. By splitting the matrix
\[
\left(\begin{matrix} 0 & a_{m_{l(q)}} \\ a_{m_{l(q)}} & 0 \end{matrix}\right)
\]
into the sum of a positive and a negative rank one perturbation, a little bookkeeping shows that
\begin{equation}
\sum_{k:\, x_{n|q,k}\in(\al_j,\be_j)} g\bigl(x_{n|q,k}\bigr)\leq
N \biggl( g(c_j)+\sum_{k:\, x_k\in(\al_j,\be_j)}g(x_k) \biggr)
\end{equation}
for some integer $N$ which is independent of $n, q$ and $j$ (cf. the proof of \cite[Prop. 3]{MR2855090}). Within any given precision, we can therefore restrict to a finite union of intervals as in \eqref{I delta} by also taking $l(n)$ sufficiently large. The desired converge of the finitely many remaining eigenvalues follows from \cite[Thm. 3.10]{MR2784484}.

If we finally let $l(n)\to\infty$, Proposition \ref{B conv} (with $z=0$) implies that
\[
\sideset{}{_k}\sum g\bigl( x_{l(n),k}\bigr) \longrightarrow \sideset{}{_j}\sum g\bigl(x_j^\pl\bigr).
\]
Since $\tau_{l(n)}\to\tau$ and because of \eqref{liminf}, the proof is clear.
\end{proof}

\begin{proof}[Proof of Theorem \ref{Jost asympt}]
It follows directly from Proposition \ref{B conv} that the Blasch\-ke parts converge locally uniformly on $\DD$. To show that the Szeg\H{o} parts converge too, it suffices to prove that
\[
\log f_{m_l}\bigl(\cm(e^{i\theta})\bigr)\mfrac{d\theta}{2\pi}
\,\xrightarrow[]{\, w \,}\,
\log f^\pl\bigl(\cm(e^{i\theta})\bigr)\mfrac{d\theta}{2\pi}
\; \mbox{ as $l\to\infty$}.
\]
As in the proof of \cite[Prop. 6.3]{MR2784484}, this is a consequence of Proposition \ref{Szego parts}.
\end{proof}

\section{Jost solutions and polynomial asymptotics}
\label{sec5}
Given a matrix $J=\{a_n,b_n\}_{n=1}^\infty$ for which $\si_\ess(J)=\large\E$, we introduce the so-called \emph{Weyl solution} by
\begin{equation}
W_n(z)=-\bigl< \de_n, (J-\cm(z))^{-1}\de_1\bigr>,
\quad \cm(z)\notin\si(J).
\end{equation}
Clearly, $W_1$ is nothing but the $M$-function for $J$ (i.e., minus the $m$-function lifted by the covering map to a meromorphic function on $\DD$). More generally,
\begin{equation}
W_n(z)=M(z)\cdot a_1M_1(z)\cdots a_{n-1}M_{n-1}(z)
\end{equation}
and since the poles of $M_k$ coincide with the zeros of $M_{k-1}$, the only poles of $W_n$ are those of $M$.
The term `solution' comes from the fact that $W_n$ obeys the three-term recurrence relation
\begin{equation}
\label{weyl rec}
\cm\, W_{n}=a_n W_{n+1}+b_nW_n+a_{n-1}W_{n-1}
\end{equation}
for all $n\geq 2$.
\begin{definition}
For a measure $d\mu$ in the Szeg\H{o} class for $\large\E$, we define the \emph{Jost solution} by
\begin{equation}
\label{Jost sol}
u_n(z;d\mu)=u(z;d\mu)\,W_n(z), \quad n\geq 1.
\end{equation}
\end{definition}
Since the Blaschke part of $u$ has zeros at the poles of $M$, $u_n$ is analytic on $\DD$.
By definition of $u$ and due to the inner-outer factorization of $M$ (see, e.g., \cite[Sect. 3]{MR2855090}), we also have
\begin{equation}
\label{un}
u_n(z;d\mu)=a_n^{-1}B(z)^n u(z;d\mu_n),
\end{equation}
where $d\mu_n$ is the spectral measure of the stripped matrix $J\vert_n$. Hence,
\begin{equation}
\label{uW}
u(z;d\mu_n)=\frac{a_nW_n(z)}{B(z)^{n}}u(z;d\mu)
\end{equation}
and it readily follows that
\begin{equation}
\label{linearize}
\chi(J\vert_n)=\chi_{\substack{\\ 0}}^{-n}\chi(J),
\end{equation}
where $\chi_{\substack{\\ 0}}$ is the character of $B$. We can interpret \eqref{linearize} by saying that the map $J\mapsto\chi(J)$ `linearizes' coefficient stripping. In particular, the left shift on $\T$ corresponds to multiplying by the inverse of the character of $B$ under the Abel map.

Now, let $\F\subset\DD$ be a fundamental set for $\Ga_\E$ which contains the point $z=0$ and denote by $\F^\nt$ its interior. Recall that the covering map $\cm$ is a meromorphic bijection of $\F^\nt$ onto $\overline\C\setminus[\al,\be]$.
When $J'$ belongs to $\T$, we denote by $d\mu'$ the spectral measure of $J^+$ (i.e., the restriction to $\ell^2(\N)$ of $J'$).
Since $J^+$ has no eigenvalues outside $[\al,\be]$, the Jost function $u(z;d\mu')$ is nonvanishing on $\F^\nt$. Moreover, as will be useful shortly, the sequence $u(z;d\mu'_n)$ is uniformly bounded away from zero (and infinity) on compact subsets of $\F^\nt$. Otherwise we could pass to a subsequence of matrices whose limiting Jost function would have a zero (or a pole) in $\F^\nt$, violating Theorem \ref{Jost asympt}.

\begin{proposition}
\label{u limit}
Let $d\mu$ be a measure in $\Sz(\large\E)$ and $J$ the associated Jacobi matrix.
If \eqref{limit} holds for some $J'$ in $\T$, then
\begin{equation}
\label{Jost limit}
{u_n(z;d\mu)}/{u_n(z;d\mu')} \ra 1 \, \mbox{ as $n\to\infty$}
\end{equation}
uniformly on compact subsets of $\F^\nt$.
\end{proposition}
\begin{proof}
Because of \eqref{limit}, $J$ and $J^+$ have the same right limits. Hence, Theorem \ref{Jost asympt} implies that
\begin{equation}
\label{Jost diff}
\bigl\vert u(z;d\mu_n)-u(z;d\mu_n')\bigr\vert \to 0
\end{equation}
uniformly on compact subsets of $\DD$. Since $u(z;d\mu_n')$ is uniformly bounded away from zero on any compact subset of $\F^\nt$, we have
\[
\biggl\vert \frac{a_n}{a_n'}\frac{u_n(z;d\mu)}{u_n(z;d\mu')}-1\biggr\vert
=\biggl\vert\frac{u(z;d\mu_n)}{u(z;d\mu_n')}-1\biggr\vert \longrightarrow 0
\]
locally uniformly on $\F^\nt$. The result follows by noting that $a_n/a_n'\to 1$.
\end{proof}

The key result to be proven is this section is
\begin{theorem}
\label{Sz asy}
Let $d\mu$ be a measure in $\Sz(\large\E)$ and $J$ the associated Jacobi matrix.
If $J'$ is the unique element in $\T$ for which \eqref{limit} holds, then
\begin{equation}
\label{P limit}
\frac{P_n(\cm(z),d\mu)}{P_n(\cm(z),d\mu')}
\longrightarrow\frac{u(z;d\mu)}{u(z;d\mu')} \; \mbox{ as $n\to\infty$}
\end{equation}
uniformly on compact subsets of $\F^\nt$.
\end{theorem}
The above theorem immediately implies Theorem \ref{Sz asympt} from the introduction and it enables us to describe the asymptotic behavior of $P_n(x,d\mu)$ if we can control the polynomials associated with points in $\T$.

Set $a_0=1$, $u_0=u$, and $P_{-1}=0$. Then the Jost solution $u_n$ and the lifted polynomials $P_{n-1}\circ\cm$ both satisfy the recurrence relation \eqref{weyl rec} for $n\geq 1$. Since $u_n$ is $\ell^2$ at $+\infty$ and the Wronskian of $u_n$ and $P_{n-1}$ is given by
\[
\Wr(z)=a_n\bigl(u_{n+1}P_{n-1}-u_nP_n\bigr)\raisebox{-0.01cm}{$\big\vert_{n=0}$}=-u,
\]
we have
\begin{equation}
\label{Green}
G_{nm}(z;J)=-\bigl< \de_n, (J-\cm(z))^{-1}\de_m \bigr>
=\frac{u_n(z;d\mu) P_{m-1}(\cm(z),d\mu)}{u(z;d\mu)}
\end{equation}
for all $z\in\F^\nt\setminus\{p_k\}$ (and $z\neq 0$ if $m>n+1$). Here, as in Section \ref{sec3}, the $p$'s account for the zeros of $u$ in $\F$. Combining \eqref{Green} for $J$ and $J^+$ (where $a_0'u_0=u$), it follows that
\begin{equation}
\frac{P_{n-1}(\cm(z),d\mu)}{P_{n-1}(\cm(z),d\mu')}
=\frac{G_{nn}(z;J)}{G_{nn}(z;J^+)}\frac{u_n(z;d\mu')}{u_n(z;d\mu)}\frac{u(z;d\mu)}{u(z;d\mu')}
\end{equation}
for all $z\in\F^\nt$. Hence, Theorem \ref{Sz asy} is an easy consequence of Proposition \ref{u limit} and
\begin{proposition}
\label{Green conv}
With $J$ and $J'$ as in Theorem \ref{Sz asy}, we have
\begin{equation}
\label{Gnn}
{G_{nn}(z;J)}/{G_{nn}(z;J^+)}\to 1 \, \mbox{ as $n\to\infty$}
\end{equation}
uniformly on compact subsets of $\F^\nt\setminus(\{p_k\}\cup\{0\})$.
\end{proposition}
Note that since both sides of \eqref{P limit} are analytic on $\F^\nt$, the uniform convergence in Theorem \ref{Sz asy} readily extends to compact subsets of $\F^\nt$. We shall prove Proposition \ref{Green conv} shortly, but first we give a more detailed description of the $P_n$'s associated with points in $\T$.

Let $J'\in\T$ be given. Following the notation of Section \ref{sec3}, we denote by $d\mu_n'$ the spectral measure of $J_n^+$ and use the relation
\begin{equation}
\label{un'}
a_n' u_n(z;d\mu')=B(z)^n u(z;d\mu_n')
\end{equation}
to define $u_n:=u_n(z;d\mu')$ for all $n\in\Z$, compare with \eqref{un}. In this way, the recurrence relation
\begin{equation}
\label{rec rel}
\cm u_n=a_n'u_{n+1}+b_n'u_n+a_{n-1}'u_{n-1}
\end{equation}
becomes valid for all integer values of $n$. Note that $u_n$ has a zero of order $n$ (or a pole of order $-n$) at every zero of $B$.

Besides $J'$, consider also the reflected matrix $J'^{,r}$ given by
\begin{equation}
\label{ab r}
\begin{aligned}
a_n'^{,r}&=a'_{-n-1}, \; \\
b_n'^{,r}&=b'_{-n}, \;
\end{aligned}
\bigg\}
\;\; n\in\Z.
\end{equation}
Since $\T$ is invariant under reflections of the Jacobi parameters, $J'^{,r}$ belongs to $\T$ and \eqref{rec rel} therefore holds with \raisebox{-0.2 em}{$'$} replaced by \raisebox{-0.2 em}{$'^{,r}$} everywhere. Equivalently,
\begin{equation}
\cm u_{-n}^{(r)}=a_n'u_{-(n+1)}^{(r)}+b_n'u_{-n}^{(r)}+a_{n-1}'u_{-(n-1)}^{(r)},
\end{equation}
where $u_n^{(r)}$ is short for $u_n(z;d\mu'^{,r})$. Hence $u_n$ and $u_{-n}^{(r)}$ satisfy the same recurrence relation for every $z\in\DD\setminus\{\ga(0)\}_{\ga\in\Ga_\E}$. As the former is $\ell^2$ at $+\infty$ and the latter is unbounded when $n\to\infty$, these two solutions are linearly independent for all $z\in\F^\nt\setminus\{0\}$.
Recalling that $P_{n-1}(\cm(z),d\mu')$ also satisfies \eqref{rec rel}, we deduce that
\begin{equation}
\label{Pcomb}
P_{n-1}=\frac{\Wr\bigl(P_{n-1}, u_n\bigr)u_{-n}^{(r)}-\Wr\bigl(P_{n-1}, u_{-n}^{(r)}\bigr)u_n}{\Wr\bigl(u_{-n}^{(r)},u_n\bigr)}
\end{equation}
for $n\geq 1$. By use of \eqref{un'}, the numerator can be written as
\begin{equation}
\label{num}
u(z;d\mu')\frac{B(z)^{-n}}{a_{n-1}'}u(z;d\mu'^{,r}_{-n})-
\frac{a_0'}{a_{-1}'}u(z;d\mu'^{,r})\frac{B(z)^n}{a_n'}u(z;d\mu'_n).
\end{equation}
Moreover, the denominator is given by
\begin{equation}
\label{den}
\Wr\bigl(u_{-n}^{(r)},u_n\bigr)=
B(z)^{-1}\medop{\sideset{}{_j}\prod} B(z,p_j^\plm)-B(z)\medop{\sideset{}{_j}\prod} B(z,z_j^\plm),
\end{equation}
where $p_j^\plm$ and $z_j^\plm$ are the unique points in $\F$ for which $\{\cm(p_j^\plm)\}$ and $\{\cm(z_j^\plm)\}$ form the complete sets of poles
and zeros of $m^\pl$ and $m^\mn$ in $\R\setminus\large\E$.
To see this, start by noting that $d\mu'^{,r}_{-1}$ is nothing but the spectral measure of the matrix $J^-$ defined in Section \ref{sec2}.
By \eqref{un'} and \eqref{uW}, we have
\begin{align}
\notag
\Wr\bigl(u_{-n}^{(r)},u_n\bigr)&=
a_n'\bigl(u_{-n-1}^{(r)} u_n-u_{-n}^{(r)}u_{n+1}\bigr)\raisebox{-0.01cm}{$\big\vert_{n=0}$} \\
&= \frac{u(z;d\mu')u(z;d\mu'^{,r}_{-1})}{a_0'B(z)}\Bigl\{1-(a'_0)^2M^+(z)M^-(z)\Bigr\}.
\end{align}
Due to the reflectionless condition \eqref{refl m}, the product of $M$-functions reduces to a ratio of Blaschke products and the Szeg\H{o} parts
of the Jost functions cancel one another. Thus we obtain \eqref{den}.

Our considerations lead to the following
\begin{proposition}
\label{P'}
Suppose that $J'\in\T$ and let $J'^{,r}$ be given by \eqref{ab r}. If $d\mu'^{,r}$ denotes the spectral measures of $J'^{,r}$ restricted to $\ell^2(\N)$, then
\begin{equation}
\label{Jminus}
\biggl\vert
a_{n}'B(z)^{n}P_{n}(\cm(z),d\mu')-
\frac{u(z;d\mu_{-n-1}'^{,r})u(z;d\mu')}{\prod_j B(z,p_j^\plm)-B(z)^2\prod_jB(z,z_j^\plm)}
\biggr\vert \longrightarrow 0 \, \mbox{ as $n\to\infty$}
\end{equation}
locally uniformly on $\F^\nt$. In particular, $B(z)^n P_n(\cm(z),d\mu')$ is uniformly bound\-ed away from zero and infinity on compact subsets of $\F^\nt$.
\end{proposition}
Put into words, \eqref{Jminus} says that the asymptotic behaviour of $P_n(\,\cdot\,,d\mu')$ is controlled by $a_n'B^n$ and the Jost function for $J^-_n$.
The remaining factor involving the Jost function for $J^+$ and certain Blaschke products is independent of $n$.
\begin{proof}
Multiply the expression for $P_{n-1}$ in \eqref{Pcomb} by $a_{n-1}'B(z)^{n-1}$ and rewrite the quotient by use of \eqref{num}--\eqref{den}. The result follows by noting that $B(z)^n\to 0$ as $n\to\infty$ while $u(z;d\mu'_n)$ is uniformly bounded on compact subsets of $\F^\nt$.
\end{proof}

By combining Theorem \ref{Sz asy} and Proposition \ref{P'}, we can get a description of $P_n(\,\cdot\,,d\mu)$ when $d\mu$ lies in $\Sz(\large\E)$
and $n$ is large.
It follows from \eqref{Jminus} that
\[
\frac{a_n'B(z)^nP_n(\cm(z),d\mu')}{u(z;d\mu_{-n-1}'^{,r})}
\longrightarrow \frac{u(z;d\mu')}{\prod_j B(z,p_j^\plm)-B(z)^2\prod_jB(z,z_j^\plm)}
\]
uniformly on compact subsets of $\F^\nt$. By \eqref{P limit} and since $a_n/a_n'\to 1$, we arrive at 
\begin{corollary}
\label{Pn}
For $d\mu$ in the Szeg\H{o} class for $\large\E$ and with notation as above, we have
\begin{equation}
\label{Pn asymp}
P_n(\cm(z),d\mu) \,\sim\, \frac{u(z;d\mu'^{,r}_{-n-1})}{a_n B(z)^n}
\frac{u(z;d\mu)}{\prod_j B(z,p_j^\plm)-B(z)^2\prod_jB(z,z_j^\plm)}
\, \mbox{ as $n\to\infty$}
\end{equation}
locally uniformly on $\F^\nt$.
\end{corollary}
The above result describes the large $n$ behaviour of $P_n(x,d\mu)$ for $x\in\C\setminus\cvh(\large\E)$.
When $\large\E$ is the interval $[-2,2]$, the formula in \eqref{Pn asymp} reduces to
\[
z^n P_n(z+1/z,d\mu) \,\sim\, \frac{u(z;d\mu)}{1-z^2} 
\]
which is consistent with \cite{MR1844996}.
There is no similar description on the set $\large\E$, where the zeros of the polynomials accumulate. As a matter of fact, one can only obtain $L^2$ asymptotics on the spectrum. For $x\in\Om$, we denote by $\icm(x)$ the unique point in $\F$ such that $\cm(\icm(x))=x$ and define $\icm(t):=\icm(t-i0)$ for $t\in\large\E$.
Following the approach of \cite[Sect. 8]{MR2784484} and with $d\mu=f(t)dt+d\mu_{\s}$, one is led to
\begin{equation}
\int_\E\,\bigl\vert P_n(t,d\mu)-\im\bigl\{u(\icm(t);d\mu) u_{n+1}(\icm(t);d\mu')\bigr\}\bigr\vert^2 f(t)dt \to 0
\end{equation}
and
\begin{equation}
\int_\R \vert P_n(x)\vert^2 d\mu_\s(x) \to 0
\end{equation}
as $n\to\infty$. See also \cite{MR1981915}.
For homogeneous sets, this type of asymptotics led Volberg and Yuditskii \cite{MR1896882} to construct a scattering theory for Jacobi matrices with an almost periodic background. In this theory, the elements in $\T$ are exactly the ones with vanishing reflection coefficients.


We conclude the section with the proof of Proposition \ref{Green conv}.
\begin{proof}
By the second resolvent identity, we have
\begin{align*}
G_{nn}(z;J^+)-G_{nn}(z;J)&=
\bigl< \de_n, (J-\cm(z))^{-1} (J-J^+) (J^+-\cm(z))^{-1}\de_n\bigr> \\[0.2em]
&=\sum_{k,m} G_{nk}(z;J) (J-J^+)_{km} G_{mn}(z;J^+).
\end{align*}
As the $km$ entry of $J-J^+$ is zero whenever $\vert k-m \vert>1$, the double sum reduces to
\[
\sum_k G_{nk}(z;J)\sum_{l=0,\pm 1}(J-J^+)_{k,k+l}\,G_{k+l,n}(z;J^+).
\]
Given a compact subset $K\subset\F^\nt\setminus(\{p_k\}\cup\{0\})$, there are constants $C, \de>0$ so that the Combes--Thomas estimate
\[
\big\vert G_{nk}(z;J)\big\vert,\, \big\vert G_{kn}(z;J^+)\big\vert\leq Ce^{-\de\vert k-n\vert}
\]
holds for all $z\in K$. Since $(J-J^+)_{k,k+l}\to 0$ as $k\to\infty$, it follows that
\[
\big\vert G_{nn}(z;J)-G_{nn}(z;J^+)\big\vert\to 0 
\]
uniformly on $K$. According to \eqref{Green} and \eqref{un}, we have
\[
G_{nn}(z;J^+)=\frac{u(z;d\mu_n')}{a_n' u(z;d\mu')}B(z)^nP_{n-1}(\cm(z),d\mu').
\]
Hence, $G_{nn}(z;J^+)$ is bounded below on $K$ by Proposition \ref{P'} and the remark prior to Proposition \ref{u limit}.
This proves \eqref{Gnn}.
\end{proof}

\small{
\noindent\emph{Jacob S. Christiansen} \\
\sc Centre for Mathematical Sciences \\
\phantom{.} \;\sc Lund University \\
\phantom{.} \;\;\;\;Box 118, 22100 Lund, Sweden \\
\phantom{.} \qquad\tt email:\,stordal@maths.lth.se}

\end{document}